\documentclass[12pt]{amsart}
\usepackage{amssymb}
\usepackage{lscape}
\usepackage[headings]{fullpage}
\usepackage{amsfonts}
\usepackage{paralist}
\usepackage{xspace}
\usepackage{euscript}
\usepackage{color,tikz}
\usetikzlibrary{decorations.pathreplacing,calc}

\usepackage{epigraph}
\usepackage[all]{xy}
\usepackage[normalem]{ulem}
\usepackage[colorlinks,pagebackref=true]{hyperref}

\makeatletter
\@addtoreset{equation}{section}
\def\theequation{\thesection.\@arabic \c@equation}

\def\theenumi{\@roman\c@enumi}
\def\@citecolor{blue}
\def\@linkcolor{blue}
\def\@urlcolor{blue}

\makeatother

\newtheorem{lemma}[equation]{Lemma}

\newtheorem{claim*}{Claim}

\newtheorem{theorem}[equation]{Theorem}

\theoremstyle{definition}
\newtheorem{rmk}[equation]{Remark}
\newenvironment{remark}[1][]{%
    \begin{rmk}[#1] \pushQED{\qed}}{\popQED \end{rmk}}

\newtheorem{eg}[equation]{Example}

\newtheorem{definition}[equation]{Definition}

\newtheorem{notn}[equation]{Notation}

\def\<{\langle}
\def\>{\rangle}

\newcommand{\Tor}{\operatorname{Tor}}

\renewcommand{\to}{\longrightarrow}

\newcommand{\kk}{\mathbf{k}}

\newcommand{\bF}{\mathbf{F}}

\newcommand{\cO}{{\mathcal{O}}}
\newcommand{\PP}{\mathbb{P}}
\newcommand{\QQ}{\mathbb{Q}}

\newcommand{\VV}{\mathbf{V}}

\newcommand{\ZZ}{\mathbb{Z}}

\newcommand{\Sym}{\operatorname{Sym}} 

\newcommand{\defi}[1]{{\bfseries\upshape #1}}

\newcommand{\nb}{\overline{\beta}}

\title{The Betti table of a high degree curve is asymptotically pure}
\author{Daniel Erman}
  \address{Department of Mathematics, University of Wisconsin, Madison, WI, 53706}
   \email{{\tt derman@math.wisc.edu}}
\thanks{Research of the author partially supported by the Simons Foundation.}

\date{\today}

\begin{document}
\date{\today}
\maketitle

\begin{abstract}
We prove that asymptotically in the degree, the main term of the Boij--S\"oderberg decomposition of a high degree curve is a single pure diagram that only depends on the genus of the curve.  This answers a question of Ein and Lazarsfeld in the case of curves.
\end{abstract}

\vspace{.5cm}

\begin{center}
{Dedicated to Rob Lazarsfeld on the occasion of his
sixtieth birthday.}
\end{center}

\section*{Introduction}
Syzygies can encode subtle geometric information about an algebraic variety, with the most famous examples coming from the study of smooth algebraic curves.  Though little is known about the syzygies of higher dimensional varieties, Ein and Lazarsfeld have shown that at least the asymptotic behavior is uniform~\cite{ein-lazarsfeld-asymptotic}.  More precisely, given a projective variety $X\subseteq \PP^n$ embedded by the very ample bundle $A$, Ein and Lazarsfeld ask:  which graded Betti numbers are nonzero for $X$ reembedded by $dA$?  They prove that, asymptotically in $d$, the answer (or at least the main term of the answer) only depends on the dimension of $X$.

Boij--S\"oderberg theory~\cite{eis-schrey1} provides refined invariants of a graded Betti table, and it is natural to ask about the asymptotic behavior of these Boij--S\"oderberg decompositions.  In fact, this problem is explicitly posed by Ein and Lazarsfeld~\cite[Problem~7.4]{ein-lazarsfeld-asymptotic}, and we answer their question for smooth curves in Theorem~\ref{thm:main}.  

Fix a smooth curve $C$ and a sequence $\{A_d\}$ of increasingly positive divisors on $C$.  We show that, as $d\to \infty$, the Boij--S\"oderberg decomposition of the Betti table of $C$ embedded by $|A_d|$ is increasingly dominated by a single pure diagram that depends only on the genus of the curve.  The proof combines an explicit computation about the numerics of pure diagrams with known facts about when an embedded curve satisfies Mark Green's $N_p$-condition.

\section{Setup}
We work over an arbitrary field $\kk$.  Throughout, we will fix a smooth curve $C$ of genus $g$ and a sequence $\{A_d\}$ of line bundles of increasing degree.  Since we are interested in asymptotics, we assume that for all $d$, $\deg A_d \geq 2g+1$.
Let $r_d := \dim H^0(C,A_d)-1=\deg A_d-g$ so that the complete linear series $|A_d|$ embeds $C\subseteq \PP^{r_d}$. For each $d$, we consider the homogeneous coordinate ring $R(C,A_d):=\oplus_{e\geq 0} H^0(C,eA_d)$ of this embedding.  We may then consider $R(C,A_d)$ as a graded module over the polynomial ring $\Sym( H^0(C,A_d))$.  

If $\bF=[\bF_0\gets\bF_1\gets \dots \gets \bF_n\gets 0]$ 
is a minimal graded free resolution of $R(C,A_d)$, then we will use $\beta_{i,j}(\cO_C,A_d)$ to denote the
number of minimal generators of $\bF_i$ of degree $j$. Equivalently, we have
\[
\beta_{i,j}(\cO_C,A_d) = \dim_{\kk} \Tor_i^{\Sym( H^0(C,A_d))}(R(C,A_d),\kk)_j.
\]
We define the \defi{graded
Betti table} $\beta(\cO_C,A_d)$ as the vector with coordinates
$\beta_{i,j}(\cO_C,A_d)$ in the vector space $\VV = \bigoplus_{i=0}^n \bigoplus_{j\in
\ZZ}\QQ$.

We use the standard Macaulay2 notation for displaying Betti tables, where
\[
\beta=
\begin{pmatrix} \beta_{0,0} & \beta_{1,1} & \beta_{2,2} &\dots\\
\beta_{0,1} & \beta_{1,2} &\beta_{2,3}&\dots  \\
\beta_{0,2} &  \beta_{1,3}&\beta_{2,4}  &\dots  \\
\vdots &\vdots &\vdots & \ddots
\end{pmatrix}.
\]

Boij--S\"oderberg theory focuses on the rational cone spanned by all graded Betti tables in $\VV$.  The extremal rays of this cone correspond to certain \defi{pure diagrams}, and hence every graded Betti table can be written as a positive rational sum of pure diagrams; this decomposition is known as a \defi{Boij--S\"oderberg decomposition}.  For a good introduction to the theory, see either~\cite{eis-schrey-icm} or~\cite{floystad-expository}.  We introduce only the notation and results that we need.

For a given $d$ and some $i\in [0,g]$, we define the (degree) sequence $\mathbf e=\mathbf e(i,d):=(0,2,3,4,\dots, r_d-i,r_d-i+2,r_d-i+3,\dots, r_d+1)\in \ZZ^{r_d-1}$, and we define the pure diagram $\pi_{i,d} \in \VV$ by the formula:
\begin{equation}\label{eqn:piid}
\beta_{p,q}(\pi_{i,d})=\begin{cases}
(r_d-1)! \cdot \left( \prod_{\ell \ne p} \frac{1}{| \mathbf e_{\ell} - \mathbf e_p |} \right ) & \text{ if } p\in [0,r_d-1] \text{ and } q=\mathbf e_p\\
0 & \text{else}.
\end{cases}
\end{equation}
Note that the shape of $\pi_{i,d}$ is the following, where $*$ indicates a nonzero entry:
\[
\pi_{i,d} = 
 \bordermatrix{
\ &0&1&\dots&r_d-i-1&r_d-i&\dots&r_d-1\cr
\ &*&0 &\dots&0&0 &\dots&0\cr
\ &0&*&\dots &*&0&\dots &0\cr
\ &0&0&\dots &0&*&\dots &*\cr
}.
\]
It turns out that these are the only pure diagrams that appear in the Boij--S\"odergberg decomposition of the Betti tables $\beta(C,A_d)$ (see Lemma~\ref{lem:pures} below).

We next recall the notion of a {(reduced) Hilbert numerator}, which will be central to our proof.  If $S=\kk[x_0,\dots, x_n]$ is a polynomial ring, and $M$ is a graded $S$-module, then the \defi{Hilbert series} of a finitely generated, graded module $M$ is the power series $\operatorname{HS}_M(t):=\sum_{i\in \ZZ} \dim_{\kk} M_i\cdot t^i \in \QQ[[t]]$.  The Hilbert series can be written uniquely as a rational function of the form
\[
\operatorname{HS}_M(t) = \frac{\operatorname{HN}_M(t)}{(1-t)^{\dim M}}
\]
and we define the \defi{Hilbert numerator} of $M$ as the polynomial $\operatorname{HN}_M(t)$.  The \defi{multiplicity} of $M$ is $\operatorname{HN}_M(1)$.

As is standard in Boij--S\"oderberg theory, we allow formal rescaling of Betti tables by rational numbers.  Note the Hilbert numerator is invariant under modding out by a regular linear form or adjoining an extra variable; also, the Hilbert numerator is computable entirely in terms of the graded Betti table (see~\cite[\S1]{eis-syzygies}).  Similar statements hold for the codimension of a module.  Thus we may and do formally extend the notions of Hilbert numerator, codimension, and multiplicity to all elements of the vector space $\VV$.

\begin{lemma}\label{lem:mult1}
For any $i,d$, the diagram $\pi_{i,d}$ has multiplicity $1$.
\end{lemma}
\begin{proof}
By~\eqref{eqn:piid} we have $\beta_{0,0}(\pi_{i,d})=\frac{(r_d-1)!}{2\cdot 3\cdots (r_d-i)\cdot (r_d-i+2)\cdots (r_d+1)}$.  Up to a positive scalar multiple, the diagram $\pi_{i,d}$ equals the graded Betti table of a Cohen--Macaulay module by~\cite[Theorem~0.1]{eis-schrey1}.  Then by Huneke and Miller's multiplicity computation for Cohen--Macaulay modules with a pure resolution\footnote{Strictly speaking, Huneke and Miller's computation is for graded algebras.  But by including a $\beta_{0,0}$ factor, the argument goes through unchanged for a graded Cohen--Macaulay module generated in degree $0$ and with a pure resolution.}~\cite[Proof of Theorem~1.2]{huneke-miller}, it follows that the multiplicity of $\pi_{i,d}$ equals
\begin{align*}
\beta_{0,0}(\pi_{i,d}) \cdot \frac{2\cdot 3\cdots (r_d-i)\cdot (r_d-i+2)\cdots (r_d+1)}{(r_d-1)!}&=\beta_{0,0}(\pi_{i,d})  \cdot (\beta_{0,0}(\pi_{i,d}) )^{-1} \\
&=1.
\end{align*}
\end{proof}

\section{Main result and proof}
To make sensible comparisons between the graded Betti tables $\beta(\cO_C,A_d)$ for different values of $d$, we will rescale by the degree of the curve so that we are always considering Betti tables of (formal) multiplicity equal to $1$.  Namely, we define
\[
\overline{\beta}(\cO_C,A_d) := \tfrac{1}{\deg A_d}\cdot \beta(\cO_C,A_d).
\]
The Boij--S\"oderberg decomposition of $\nb(\cO_C,A_d)$ has a relatively simple form.
\begin{lemma}\label{lem:pures}
For any $d$, the Boij--S\"oderberg decomposition of $\nb(\cO_C,A_d)$ has the form
\begin{equation}\label{eqn:main expression}
\nb(\cO_C,A_d) = \sum_{i=0}^{g} c_{i,d} \cdot \pi_{i,d},
\end{equation}
where $c_{i,d}\in \QQ_{\geq 0}$ and $\sum_i c_{i,d}=1$.
\end{lemma}
The above lemma shows that the number of potential pure diagrams appearing in the decomposition of $\nb(C,A_d)$ is at most $g+1$.  Note that the precise number of summands with a nonzero coefficient is closely related to Green and Lazarsfeld's Gonality Conjecture~\cite[Conjecture~3.7]{green-lazarsfeld}, and hence will vary even among curves of the same genus.   However our main result, which we now state, shows that this variance plays a minor role in the asymptotics:
\begin{theorem}\label{thm:main}
The Betti table $\nb(\cO_C,A_d)$ converges to the pure diagram $\pi_{g,d}$ in the sense that
\[
c_{i,d}\to \begin{cases} 0 & i\ne g \\1 & i = g  \end{cases} \qquad \text{ as \qquad $d\to \infty$}.
\]
\end{theorem}
\noindent In particular, the limiting pure diagram only depends on the genus of the curve.  A nearly equivalent statement of the theorem is: asymptotically in $d$, the main term of the Boij--S\"oderberg decomposition of the (unscaled) Betti table $\beta(C,A_d)$ is the $\pi_{g,d}$ summand. 

\begin{proof}[Proof of Lemma~\ref{lem:pures}]
Since the homogeneous coordinate ring of $C\subseteq \PP^{r_d}$ is Cohen--Macaulay (see ~\cite[~\S8A]{eis-syzygies} for a proof and the history of this fact), it follows from \cite[Theorem~0.2]{eis-schrey1} that $\nb(\cO_C,A_d)$ can be written as a positive rational sum of pure diagrams of codimension $r_d-1$.
Since $C\subseteq \PP^{r_d}$ satisfies the $N_p$ condition for $p=r_d-g-1$ by~\cite[Theorem~4.a.1]{green-koszul-1}, it follows that the shape of $\beta(\cO_C,A_d)$ is:
\[
 \bordermatrix{
\ &0&1&2&\dots &r_d-g-1&r_d-g&\dots & r_d-1\cr
\ &*&-&\dots &-&-&-&\dots &-\cr
\ &-&*&*&\dots &*&*&\dots&*\cr
\ &-&-&-&\dots &-&*&\dots& *\cr
}.
\]
Thus the pure diagrams $\pi_{i,d}$ for $i=0,1,\dots, g$ are the only diagrams that can appear in the Boij--S\"oderberg decomposition of $\nb(C,A_d)$, and so we may write
\[
\nb(\cO_C,A_d) = \sum_{i=0}^{g} c_{i,d} \cdot \pi_{i,d}
\]
with $c_{i,d}\in \QQ_{\geq 0}$.  The (formal) multiplicity of $\nb(C,A_d)$ is $1$ by construction, and the same holds for the $\pi_{i,d}$ by Lemma~\ref{lem:mult1}, so it follows that $\sum c_{i,d}=1$.
\end{proof}

\begin{lemma}
The Hilbert numerator of the pure diagram $\pi_{i,d}$ is
\[
\operatorname{HN}_{\pi_{i,d}}(t)=\left(\frac{r_d-i+1}{r_d(r_d+1)}\right) t^0+ \left(\frac{(r_d-1)(r_d-i+1)}{r_d(r_d+1)}\right) t^1 + \left(\frac{i}{r_d+1}\right) t^2.
\]
The Hilbert numerator of the rescaled Betti table $\nb(\cO_C,A_d)$ is
\[
\left(\frac{1}{r_d+g}\right) t^0+ \left(\frac{r_d-1}{r_d+g}\right) t^1 + \left(\frac{g}{r_d+g}\right) t^2.
\]
\end{lemma}
\begin{proof}
We prove the first statement by direct computation.  Since $\pi_{i,d}$ represents, up to scalar multiple, the Betti table of a Cohen--Macaulay module $M$, we may assume by Artinian reduction that the module $M$ has finite length.  For a finite length module, the Hilbert numerator equals the Hilbert series.  Since the Betti table $\pi_{i,d}$ has $2$ rows, it follows that the Castelnuovo--Mumford regularity of $M$ equals $2$ (except for $\pi_{0,d}$ which has regularity $1$).  The coefficient of $t^0$ is thus the value of the Hilbert function in degree $0$, which is the $0$th Betti number of the pure diagram $\pi_{i,d}$.  By \eqref{eqn:piid}, this equals
\begin{align*}
\beta_{0,0}(\pi_{i,d}) &=\frac{(r_d-1)!}{2\cdot 3\cdots (r_d-i)\cdot (r_d-i+2)\cdots (r_d+1)}=\frac{r_d+1-i}{r_d(r_d+1)}.
\end{align*}
Similarly, the coefficient of $t^2$ is given by the bottom-right Betti number of $\pi_{i,d}$ which is
\[
\beta_{r_d-1,r_d+1}(\pi_{i,d}) = \frac{i}{r_d+1}.
\]

Finally, since $\pi_{i,d}$ has multiplicity $1$ by Lemma~\ref{lem:mult1}, it follows that $\operatorname{HN}_{\pi_{i,d}}(1)=1$ and hence the coefficient of $t^1$ equals $1$ minus the coefficients of $t^0$ and $t^2$:
\begin{align*}
1- \left(\frac{r_d+1-i}{r_d(r_d+1)}\right) - \left(\frac{i}{r_d+1} \right)
&= \frac{r_d(r_d+1) - (r_d-i+1) - i\cdot r_d}{r_d(r_d+1)}\\
&=\frac{(r_d-1)(r_d-i+1)}{r_d(r_d+1)}.
\end{align*}

For the Hilbert numerator of $\nb(\cO_C,A_d)$ statement, we note that $\deg A_d=r_d+g$, yielding
\[
\nb(\cO_C,A_d) = \tfrac{1}{r_d+g} \cdot \beta(\cO_C,A_d).
\]
As above, we can compute the $t^0$ and $t^2$ coefficients via the first and last entries in the Betti table, and these are thus $\frac{1}{r_d+g}$ and $\frac{g}{r_d+g}$ respectively (see~\cite[\S8A]{eis-syzygies}, for instance).  Since $\nb(\cO_C,A_d)$ has multiplicity $1$, the $t^1$ coefficient is again $1$ minus the $t^0$ and $t^2$ coefficients, and the statement follows immediately.
\end{proof}

\begin{proof}[Proof of Theorem~\ref{thm:main}]
Note that $r_d\to \infty$ as $d\to \infty$.  We rewrite the Hilbert numerator of $\pi_{i,d}$ as
\[
\operatorname{HN}_{\pi_{i,d}}(t) = \left(\frac{1}{r_d} + \epsilon_{0,i,d}\right) t^0+ \left(1 + \frac{1}{r_d} + \epsilon_{1,i,d}\right) t^1 + \left(\frac{i}{r_d} + \epsilon_{2,i,d}\right) t^2,
\]
where $r_d\epsilon_{j,i,d}\to 0$ as $d\to \infty$ for all $j=0,1,2$ and $i=0,\dots, g$.  For instance
\[
\epsilon_{0,i,d} = \frac{r_d-i+1}{r_d(r_d+1)} - \frac{1}{r_d}=\frac{-i}{r_d(r_d+1)}.
\]
We may similarly rewrite the Hilbert numerator of $\nb(\cO_C,A_d)$ as
\[
\left(\tfrac{1}{r_d} + \delta_{0,d}\right)t^0 + \left(1-\tfrac{g+1}{r_d} + \delta_{1,d}\right)t^1+\left(\tfrac{g}{r_d} + \delta_{2,d}\right)t^2
\]
where for $j=0,1,2$ we have $r_d\delta_{j,d}\to 0$ as $d\to \infty$.

Since the Hilbert numerator is additive with respect to the Betti table decomposition of~\eqref{eqn:main expression}, combining the above computations with our Boij--S\"oderberg decomposition from~\eqref{eqn:main expression}, we see that the $t^2$ coefficient of the Hilbert numerator of $\nb(\cO_C,A_d)$ may be written as
\[
 \tfrac{g}{r_d}+\delta_{2,d} = \sum_{i=0}^g c_{i,d}\cdot \left(\tfrac{i}{r_d}+\epsilon_{2,i,d}\right).
\]
We multiply through by $r_d$ and take the limit as $d\to \infty$.  Since $r_d\delta_{j,d}$ and $r_d\epsilon_{j,i,d}$ both go to $0$ as $d\to \infty$, this yields:
\[
g = \lim_{d\to \infty} \sum_{i=0}^g c_{i,d}\cdot i.
\]
But $c_{i,d}\geq 0$ and $\sum_{i} c_{i,d}=1$.  Hence, as $d\to \infty$, we obtain $c_{i,d}\to 0$ for all $i\ne g$ and $c_{g,d}\to 1$.
\end{proof}

\begin{remark}
 If $X$ is a variety with $\dim X>1$, then our argument fails in several important ways.  To begin with, Ein and Lazarsfeld's nonvanishing syzygy results from~\cite{ein-lazarsfeld-asymptotic} show that the number of potential pure diagrams for the Boij--S\"oderberg decomposition of $\beta(X,A_d)$ is unbounded.  
 
 Moreover, in the case of curves, the Hilbert numerator of the embedded curves converged to the Hilbert numerator of one of the potential pure diagrams; the $N_p$ condition then implied that this pure diagram had the largest degree sequence of any potential pure diagram.  Our result then followed by the semicontinuous behavior of the Hilbert numerators of pure diagrams (for a related semicontinuity phenomenon see~\cite[Monotonicity Principle, p.~758]{ees-filtering}).   Ein and Lazarsfeld's asymptotic nonvanishing results imply that, even for $\PP^2$, the limit of the Hilbert numerator will fail to correspond to an extremal potential pure diagram, and so the semicontinuity does not obviously help.
\end{remark}

\section*{Acknowledgments}
The questions considered in this paper arose in conversations with Rob Lazarsfeld, and in addition I learned a tremendous amount about these topics from him; it is a great pleasure to thank Rob Lazarsfeld for his influence and his superb mentoring.  I also thank Lawrence Ein and David Eisenbud for helpful insights and conversations related to this paper.  I thank Christine Berkesch, Frank-Olaf Schreyer, and the referee for comments that improved this paper.


\end{document}